\documentclass[10pt]{amsart}
\usepackage{amssymb}
\usepackage{bm}
\usepackage[centertags]{amsmath}
\usepackage{amsfonts}
\usepackage{amsthm}
\newtheorem{thm}{Theorem}[section]
\newtheorem{cor}[thm]{Corollary}
\newtheorem{lem}[thm]{Lemma}
\newtheorem{prop}[thm]{Proposition}

\theoremstyle{definition}
\newtheorem{defn}[thm]{Definition}
\newtheorem{Conjecture}{Conjecture}

\newtheorem{rem}[thm]{Remark}
\newtheorem{examp}[thm]{Example}
\newtheorem*{question}{Question}

\begin{document}
\title{A Hales--Jewett type property of finite solvable groups}

\author{Vassilis Kanellopoulos and Miltiadis Karamanlis}

\address{National Technical University of Athens, Faculty of Applied Sciences,
Department of Mathematics, Zografou Campus, 157 80, Athens, Greece}
\email{bkanel@math.ntua.gr}

\address{National Technical University of Athens, Faculty of Applied Sciences,
Department of Mathematics, Zografou Campus, 157 80, Athens, Greece}
\email{kararemilt@gmail.com}



%------------------------Abstract-------------------------------%
\begin{abstract}  A   conjecture of Leader, Russell and Walters  in Euclidean Ramsey theory says that a finite set is  Ramsey if and only if it is congruent to a subset  of a set whose symmetry group acts transitively.  As they have shown the ``if" direction of their conjecture follows if all finite groups have a Hales--Jewett type property. In this paper, we show that this property is satisfied in the case of finite solvable groups. Our result can be used to recover the work of  K\v{r}\'\i\v{z} in Euclidean Ramsey theory.
\end{abstract}

%--------------------------Introduction-----------------------%

\maketitle
\numberwithin{equation}{section}

\section{Introduction}\label{sec1}
\subsection{Overview}  A finite set $X$ in $\mathbb{R}^n$ is called \emph{Ramsey} if for every $r\in\mathbb{N}$ there exists a positive integer  $N=N(X,r)$ such that for every $r$-coloring of $\mathbb{R}^N$ there exists  a monochromatic subset of $\mathbb{R}^N$ which is congruent to $X$. The concept of Ramsey set was originally introduced and studied   by Erd\H{o}s, Graham, Montgomery, Rothschild,  Spencer, and Straus in  \cite{{EGMRSS73}, EGMRSS75a, EGMRSS75b}. Frankl and R\"{o}dl in \cite{FR90} proved  that any $n$-dimensional simplex  is Ramsey (see also \cite{FR04}  and  \cite{MR95}). Another important result is that of  K\v{r}\'\i\v{z} in \cite{Kr91} saying that a finite subset of $\mathbb{R}^n$ is Ramsey if  its symmetry group acts transitively and  has  a solvable subgroup  with  at most two orbits. As a consequence, he showed that all  regular polygons  and all regular polyhedra   are Ramsey. For a general survey  of Euclidean Ramsey theory see \cite{Gr18}.

A central problem in Euclidean Ramsey theory  is to determine which sets are Ramsey. In  \cite{EGMRSS73}  it  was shown that every Ramsey set is  \emph{spherical}, that is  it lies on the surface of some sphere. A  well-known  conjecture of Graham  \cite{Gr94} says that the Ramsey sets are exactly the  spherical sets.  In \cite{LRW12}, Leader, Russell and Walters   proposed an alternative conjecture, stating  that a finite set is Ramsey if and only if it is \emph{subtransitive} in the sense that  it is congruent to a subset of a finite  set whose symmetry group acts transitively.  Their conjecture is  a genuine refinement of that of  Graham's, since although every subtransitive set is spherical,  almost all four point subsets of the circle are not subtransitive (see \cite{LRW11}). Concerning the one  direction of their conjecture, saying that  a set is Ramsey if it is subtransitive,  they showed that it can be reduced to  a list of equivalent  conjectures  which are free from any geometric notion and hopefully more manageable to prove. One of these  conjectures \cite[Conjecture C]{LRW12} states that every finite group satisfies a  combinatorial statement  which resembles the Hales--Jewett theorem \cite{HJ63}.  In this paper,  we show that in the case of finite solvable groups   a stronger statement  is  satisfied. The motivation to investigate this class of groups comes from  the above mentioned work of K\v{r}\'\i\v{z}  \cite{Kr91}. To state the conjecture and
our results we will need   an algebraic notion of variable words  originally introduced by   Graham and Rothschild in \cite{GR71}, which we develop in the next subsection.

\subsection{Variable words and groups.}

First let us state some general notation. In the following  by  $\mathbb{N}$ we denote the set of all positive integers and for every $n\in\mathbb{N}$, let $[n]=\{1,\ldots,n\}$.  Also for a finite set $X$,  $|X|$ stands for  its cardinality. An \emph{action} of a group $G$ on a nonempty set $X$ is a map from $G\times X$ to $X$, denoted by $(g,x)\to gx$ such that   $ex=x$, for all $x\in X$ (where  $e$ is   the identity element of $G$) and  $h(gx)=(hg)x$, for all $x\in X$ and all $h,g\in G$. For any  $x\in X$ the set $G x=\{gx:g\in G\}$ is  an  \emph{orbit of $G$ in $X$}.   We say that  $G$ \emph{acts transitively} (or it is \emph{transitive}) if $G$ has only one orbit.

Let  $G$ be a finite group acting  on a finite set $X$. We view  $X$ as an  \emph{alphabet} and its elements as \emph{letters} and  we call the finite sequences with values in $X$, \emph{constant words}.  We also fix a set $\{v_g:g\in G\}$ of distinct \emph{variables} indexed by the set $G$,  such that $v_g\notin X$ for every $g\in G$. For every nonempty subset $H$ of $G$,  by the term $H$-\emph{variable word} \emph{over} $X$ (of \emph{length} $N$), we mean a finite sequence $W=(w_i)_{i=1}^N$ with $w_i\in X\cup\{v_h: h\in H\}$ for all $i\in[N]$, such that the set $F_h=\{i\in[N]:w_i=v_h\}$ is non empty for all $h\in H$. The positive integer  $d=\sum_{h\in H}\left|F_h\right|$ will be  called the \emph{degree} of  $W$. Given an $H$-variable word $W$ over $X$ and $x\in X$, by $W(x)$ we denote the constant word of the same length obtained by leaving the letters of $W$ unchanged and replacing each variable $v_h$ by  $hx$, where $hx$ is the result of the action of $h$ on $x$, i.e.   if $W=(w_i)_{i=1}^N$ then $W(x)=\left(w_i(x)\right)_{i=1}^N\in X^N$, where
\begin{equation}\label{EQT1.1}
	w_i(x)=	\begin{cases} 	
							w_i	& \text{if }\, w_i\in X,\\
							hx 	& \text{if }\, w_i=v_h \text{ for some } h\in H.
			\end{cases}
\end{equation}

\begin{examp}\label{EX1.1}
Let $X=[3]$, $G=S_3$ be the symmetric group of $[3]$ and $H=\{e,\tau,\tau^2\}$ be the subgroup of $S_3$ generated by the cycle $\tau=(1\ 2\ 3)$. The sequence  $W=(v_e, 1, 2, v_{\tau^2}, v_{\tau})$ is a  $H$-variable word over $X$ of length $N=5$ and degree  $d=3$. Moreover,  under the natural action of $S_3$ on  $[3]$ we have $W(1)=(1,1,2,3,2)$, $W(2)=(2,1,2,1,3)$ and $W(3)=(3, 1, 2, 2, 1)$.
\end{examp}

\begin{rem} \label{RM1.2}  
In connection with Euclidean Ramsey theory, it is worth mentioning the following geometric interpretation of the above notion of variable words. Suppose that $X$ is a finite subset of $\mathbb{R}^n$ and $G$ is the symmetry group of $X$ i.e. the set of all distance preserving maps from $X$ to $X$, with group operation the composition of functions. Let $\emptyset\neq H\subseteq G$ and $W$ be a $H$-variable word over $X$ of length $N$ and degree $d$. Then  by \eqref{EQT1.1}, it is easy to see that the map  $x\to W(x)$ is a $d^{1/2}$-\emph{dilation of} $X$ into $X^N$, that is for every  $x,x'\in X$, $\|W(x)-W(x')\|^2=d\|x-x'\|^2$, where the norm in the left (resp. right) hand-side  is the usual Euclidean norm in $\mathbb{R}^{nN}$ (resp. $\mathbb{R}^{n}$).
\end{rem}

In the following, if  $G=X$, whenever  we omit to mention any particular action of $G$ on itself, we will always mean  that the action  is the natural  operation of $G$. So in this case,  if $W$ is a $H$-variable word over $G$, then for every $g\in G$, $W(g)$ is  the constant word obtained by substituting in $W$ each variable $v_h$ by the ordinary product $hg$ of $h$ and $g$ in $G$.

For every $r\in\mathbb{N}$ and for a nonempty  set $Y$,  an $r$-\emph{coloring} of $Y$ is a map from  $Y$ to $[r]$. A subset $Z$ of $Y$ is called \emph{monochromatic} if the coloring is  constant  on $Z$.

Under the above notation, the  conjecture of Leader, Russell and Walters \cite[Conjecture C]{LRW12}  is  restated as follows.

\begin{Conjecture}\label{Conj1} \cite{LRW12}
\emph{Let $G$ be a finite  group and  $r\in\mathbb{N}$.  Then there exist positive integers\,  $d$ and $N$ such that for every $r$-coloring of $G^N$ there exist a nonempty  $H\subseteq G$ and a $H$-variable word $W$ over $G$ of length $N$ and degree $d$,  such that the set $\{W(g): g\in G\}$ is monochromatic}.
\end{Conjecture}

Assuming that Conjecture \ref{Conj1} holds and taking into account  Remark \ref{RM1.2}, it is not difficult  to show  that any  finite subset $X$  of $\mathbb{R}^n$ with  a transitive symmetry group is Ramsey (for details see  \cite[Proposition 2.1]{LRW12} or  Corollary \ref{CRL1.7} below).

In this paper we will exclusively  deal with $H$-variable words over $X$ in which for all $h\in H$ the variable $v_h$ appears the same number of times. These variable words will be called \emph{uniform}.  This notion also appears  in \cite{LRW12} in a similar  setting.  Actually, following  the circle of the equivalent conjectures in \cite{LRW12}, it can be deduced  that Conjecture \ref{Conj1} can be equivalently restated for $H=G$ and $W$ a uniform $G$-variable word over $G$.

\subsection{The main results of the paper.} 
Recall that  a group $G$ is called \emph{solvable}  if  it has  a subnormal  series $\{e\}=G_0\vartriangleleft  G_1\vartriangleleft\dots \vartriangleleft G_n=G$ such that   $G_{i}/G_{i-1}$ is   abelian for all $i\in [n]$. If $G$ is finite then $G$ is solvable if and only if it has a subnormal series such  that  all factors  are cyclic. Every abelian group is  solvable and by  the famous theorem  of Feit and Thomson \cite{FT63}, every  finite group of odd order is solvable. On the other hand, the symmetric group $S_n$ is solvable only for $n\leqslant 4$.

\begin{defn}\label{DFN1.4} 
Let  $\{e\}=G_0 \vartriangleleft   G_1\vartriangleleft  \dots \vartriangleleft  G_n=G$ be a  subnormal series with cyclic factors of a finite solvable group $G$ and  let  $p_i$ be  the order of  the factor group  $G_{i}/G_{i-1}$, for every $i\in [n]$. The number
\begin{equation}\label{EQT1.3}
\begin{split}  \prod_{i=1}^n p_i^{(p_i-1)\prod_{j>i} p_j}=p_1^{(p_1-1)\prod_{j=2}^n p_j} p_2^{(p_2-1)\prod_{j=3}^n p_j}\dots  p_n^{p_n-1}
\end{split}
\end{equation}
will be called a \emph{HJ-degree} of $G$.
\end{defn}

For example, the series   $\{e\}\vartriangleleft  A_3 \vartriangleleft S_3$ gives that the number  $3^4 \cdot 2$  is  a HJ-degree of  $S_3$. Also,  by the series  $\{e\}\vartriangleleft \{e, (1 \ 2 )(3\ 4)\} \vartriangleleft V_4 \vartriangleleft A_4 \vartriangleleft S_4$, where $V_4$  is the Klein 4-group, we get that the number $2^{19}\cdot 3^4$   is a  HJ--degree of  $S_4$. The above HJ-degrees of $S_3$ and $S_4$ are unique since the preceding series are the only  subnormal series with cyclic factors that  these groups have. In general, a solvable group may have more than one HJ-degrees. Indeed, denoting by  $C_n$  a cyclic group of order $n$,  the  series $\{e\}\vartriangleleft C_6$, $\{e\}\vartriangleleft C_2 \vartriangleleft C_6$, $\{e\}\vartriangleleft C_3 \vartriangleleft C_6$ give that the numbers $6^5$, $2^3\cdot 3^2$ and $3^4\cdot 2$ are HJ-degrees of $C_6$. The number  $n^{n-1}$  is the largest HJ-degree of $C_n$. In general, using   \eqref{EQT1.3}, it can be easily shown that refinements of a subnormal series lead to smaller HJ-degrees.

Our first main  result says  that finite solvable groups satisfy a stronger form  of  Conjecture \ref{Conj1} in the sense that  the degree $d$ of the resulting variable word may be assumed to be a HJ-degree of $G$ and hence it is independent of the number  of colors.

\begin{thm}\label{T1}
Let $G$ be a  finite solvable group,  $d$ be a HJ-degree of $G$ and  $r\in\mathbb{N}$. Then  there exists a positive integer $N$ such that  for every   $r$-coloring of $G^N$  there exists a uniform $G$-variable word $W$ over $G$ of length $N$ and degree $d$  such that the set $\left\{W(g):g\in G\right\}$ is monochromatic.
\end{thm}

Our second main  result concerns actions of finite solvable groups.

\begin{thm} \label{T2} Let $G$ be a  finite solvable group acting on a finite  set $X$,  $d$ be a HJ-degree of $G$ and  $r\in\mathbb{N}$. Then  there exists a positive integer $N$ such that for  every  $r$-coloring of $X^N$  there exists a uniform $G$-variable word $W$ over $X$  of  length $N$ and degree $d^p$, where $p$ is the number of the orbits  of $G$ in $X$, such that for every  $x\in X$, the set $\left\{W(gx):g\in G\right\}$ is monochromatic.
\end{thm}

It is clear that Theorem \ref{T2} includes  Theorem \ref{T1} as a special case. However, as we will see  the two theorems are equivalent. To illustrate the connection with Euclidean Ramsey theory, let us present  the following  consequence which is a refined form of  \cite[Theorem 4.3]{Kr91}.

\begin{cor}\label{CRL1.7} 
Let $X$ be a finite non-empty  subset of $\mathbb{R}^n$ and   $G$ be a  solvable group of isometries of $X$. Also let  $d$ be a HJ-degree of $G$ and let  $\lambda=d^{-p/2}$,  where  $p$  is the number of the orbits of $G$ in  $X$. Then for every  $r\in\mathbb{N}$ there exists a positive integer $N$  such that  for  every $r$-coloring of $\lambda X^N$ there is an isometric embedding $f:X\to \lambda X^N$ such that  for every $x\in X$,  the set $\{f(gx): g\in G\}$ is monochromatic.
\end{cor}

\begin{proof} 
Fix $r\in\mathbb{N}$. By Theorem \ref{T2} and  Remark \ref{RM1.2} (for $G$ and $d^p$ in place of $H$ and $d$), there exists a positive integer  $N$ such that  for any  $r$-coloring of $X^N$ there exists  a $d^{p/2}$-dilation  $\phi: X\to X^N$ with the property that  for every $x\in X$ the set  $\{\phi(g x): g\in G \}$ is monochromatic. Now let  $c$ be an $r$-coloring of $\lambda X^N$ and let  $c_\lambda: X^N \to [r]$ defined  by $c_\lambda(\bm{x})=c(\lambda \bm{x})$, for all $\bm{x}\in X^N$.  If  $\phi$ is the  $d^{p/2}$-dilation of $X$ into $X^N$ corresponding to $c_\lambda$, then it is easy to check that the map  $f:X\to \lambda X^N$ defined by $f(x)=\lambda \phi(x)$ is as  desired.
\end{proof}

\subsection{Organization  of the paper}  
In Section \ref{Sec2} we  define a Hales--Jewett type property  for finite groups as well as a generalization  of this property for  actions  on finite sets. Under these definitions our main results are  reduced to  three propositions.  A basic tool that we will use in  their  proofs is a variant of a well-known  lemma due to  Shelah  \cite{Sh88} which is presented  in Section \ref{Sec3}. In Section \ref{Sec4} we introduce  some  extra  notation   and the proofs are  completed in  Sections \ref{Sec5}-\ref{Sec7}. We close the paper in Section \ref{Sec8} by  stating some notes and remarks.

\section{The basic steps  of the proof of the main results.} \label{Sec2}
In the following, given  a finite group $G$ acting on a finite set $X$ and a nonempty subset $H$ of $G$, by $\mathrm{V}_{\mathrm{un}}^d(H;X)$ we will denote the set of all uniform $H$-variable words over $X$ of degree $d$. In particular, if $G=H=X$, by $\mathrm{V}_{\mathrm{un}}^d(G;G)$  we denote  the set of all uniform $G$-variable words over $G$ of degree $d$.

By  isolating  the  property of solvable groups arising from   Theorem \ref{T1} we formulate the following definition.

\begin{defn}\label{D1}
Let $G$ be a finite group and $d\in\mathbb{N}$. We will say that $G$ \emph{has the $d$-uniform Hales--Jewett property} (in short, $d$-UHJP), if for every $r\in\mathbb{N}$ there exists a positive integer $N=N(G,d,r)$ such that for every $r$-coloring  of\, $G^N$ there exists a variable word $W\in \mathrm{V}_{\mathrm{un}}^d(G;G)$ of length $N$ such that the set $\{W(g): g\in G\}$ is monochromatic.
\end{defn}

In view of Definition \ref{D1}, Theorem \ref{T1} states that if $G$ is a finite solvable group and $d$ is a HJ-degree of $G$ then $G$ has the $d$-UHJP. The trivial group $\{e\}$ has the $1$-UHJP (simply consider the variable word $W=v_e$).  Moreover,  it is easy  to see, using for instance the correspondence $g\to(g,\ldots,g)\in G^k$, that if $G$ has the $d$-UHJP then  it also has the $kd$-UHJP, for all $k\in\mathbb{N}$. The first step towards the proof of Theorem \ref{T1} is the following.

\begin{prop}\label{PRP2.2} 
If $p$ is any  positive integer and $G$ is a cyclic group of order $p$ then $G$ has the $p^{p-1}$-UHJP.
\end{prop}

For the proof of the above statement  we will use  a combinatorial argument  due to K\v{r}\'\i\v{z} \cite{Kr91}. We  introduce now  a generalization  of Definition \ref{D1} which is motivated by our second main result.

\begin{defn}\label{D2}  
Let $G$ be a finite group acting on  a finite  set $X$. Also let $H$ be a subgroup of $G$, $E$ be an equivalence relation on $X$ and $d\in \mathbb{N}$. We will say that $(H,X)$  has the  $(E,d)$-UHJP,  if  for  every $r\in\mathbb{N}$ there exists a positive integer $N=N(H,X,E,d,r)$ such that  for every $r$-coloring  of\, $X^N$ there exists a variable word $W\in \mathrm{V}_{\mathrm{un}}^{d}(H;X)$ of length $N$ such that for every $x\in X$ the set $\left\{W(x'): x' E x \right\}$ is monochromatic.
\end{defn}

If $H=X=G$ then  for simplicity we will say that $G$ has the  $(E, d)$-UHJP. If a group $G$ acts on a set $X$ then the restriction of the action to a subgroup  $H$ of $G$ yields an equivalence relation on $X$ with equivalence classes the orbits of $H$ in  $X$. We  will denote this equivalence relation by  $E_{X|H}$, that is
\begin{equation}\label{EQT2.1}
x'E_{X|H}\, x \Leftrightarrow x'\in Hx.
\end{equation}

Notice that under Definition \ref{D2} and \eqref{EQT2.1}, Theorem \ref{T2} states that if  $G$ is a finite solvable group acting on a finite set $X$ and $d$ is a HJ-degree of $G$ then $(G,X)$ has the $(E_{X|G}, d^p)$-UHJP, where  $p$ is the number of the orbits  of $G$ in $X$. One of the basic ingredients of  the proof of theorems \ref{T1} and \ref{T2} is  the next proposition.

\begin{prop}\label{PRP2.4} 
Let $G$ be a finite group acting on  a finite  set $X$ and let $H$ be a subgroup  of $G$.
If $H$ has the $d$-UHJP then $(H, X)$ has the $(E_{X|H},d^p)$-UHJP, where $p$ is the number of the orbits of $H$ in $X$.
\end{prop}

If $X=G$ (and the action  is the ordinary  multiplication of $G$)  then   $E_{G|H}$ has as equivalence classes the right cosets of $H$ in $G$.  In this case Proposition \ref{PRP2.4} takes the following form.

\begin{cor}\label{CRL2.5} 
Let $G$ be a  finite group and $H$ be a subgroup of $G$. If  $H$ has the $d$-UHJP  then $(H,G)$  has the $\left(E_{G|H}, d^p\right)$-UHJP, where  $p$ is the index of $H$ in  $G$.
\end{cor}

One of the basic properties of the class of solvable groups is that it is  closed under group extensions. Our next proposition says that the same holds for the class of finite groups having the $d$-UHJP for some $d\in\mathbb{N}$.  Recall that if   $H$  and $K$ are groups then  an \emph{extension of $K$ by $H$}  is a group $G$ along with  a surjective homomorphism $\pi: G\to K$ and an injective homomorphism $\iota: H\to G$  such that the image of $\iota$ equals the kernel of $\pi$. In particular, if $H$ is a normal subgroup of $G$ then $G$ is an extension of $G/H$ by $H$ (with $\iota$ the identity map on $H$ and $\pi$ the natural surjective homomorphism $g\to gH$ from $G$ to $G/H$).

\begin{prop}\label{PRP2.6} 
Let $H$ and  $K$ be  finite groups and let  $G$ be  an extension of $K$ by $H$.  If $H$ has the $d_H$-UHJP and $K$ has the $d_K$-UHJP then  $G$ has the $d$-UHJP, where $d=d_H^{|K|} \cdot d_K$.
\end{prop}

Assuming the above propositions we can now give  the proofs of our main results.

\begin{proof}[Proof of Theorem \ref{T1} (assuming propositions \ref{PRP2.2} and \ref{PRP2.6})] 
Let $G$ be a finite solvable group and  $d$ be a HJ-degree of $G$. Let  $\{e\}=G_0\vartriangleleft  G_1 \vartriangleleft\dots \vartriangleleft G_n=G$ be a subnormal series with cyclic factors   such that $d=\prod_{i=1}^n p_i^{(p_i-1)\prod_{j>i} p_j}$, where $p_i=|G_i/G_{i-1}|$ for all $i\in [n]$. We have to show that $G$ has the $d$-UHJP. Let $d_0=1$ and inductively define $d_i=d_{i-1}^{p_i}p_i^{p_i-1}$,  for every $i\in [n]$. As we have already mentioned, $G_0$ has the $d_0$-UHJP. Let $i\geqslant 1$ and assume that  $G_{i-1}$ has the  $d_{i-1}$-UHJP. Since $G_i/G_{i-1}$ is a cyclic group of order $p_i$, by Proposition \ref{PRP2.2}, $G_i/G_{i-1}$ has the $p_i^{p_i-1}$-UHJP. Moreover, it is clear that  $G_i$ is an extension of $G_i/G_{i-1}$  by  $G_{i-1}$ and hence, by Proposition \ref{PRP2.6} and the inductive definition of $d_i$,  $G_i$ has the $d_i$-UHJP. By induction  $G$ has the $d_n$-UHJP. It is now a matter of a simple calculation  to check  that  $d_n=d$.
\end{proof}

\begin{proof}[Proof of Theorem \ref{T2} (assuming Theorem \ref{T1} and Proposition \ref{PRP2.4})]
Let $G$ be  a finite solvable group acting on a finite set $X$ and let $d$ be  a HJ-degree of $G$. By Theorem \ref{T1},  $G$ has the $d$-UHJP. Hence, by  Proposition \ref{PRP2.4}, for $H=G$,  $(G, X)$ has the $(E_{X|G},d^p)$-UHJP, where $p$ is the number of the orbits of $G$ in $X$.  As we have already noticed this is the content of  Theorem \ref{T2}.
\end{proof}

\section{A variant of  Shelah's lemma}\label{Sec3}
For the  proof of propositions \ref{PRP2.2}, \ref{PRP2.4}, and \ref{PRP2.6} we will use a variant of a well-known  lemma due to  Shelah  used in his proof of the Hales--Jewett theorem  \cite{Sh88}  which we state below. We fix for the following   a finite group $G$ acting on a finite  set $X$, a subgroup $H$  of $G$,  an equivalence relation $E$ on $X$ and  $d\in \mathbb{N}$. For every  variable word $W\in \mathrm{V}_{\mathrm{un}}^d(H;X)$,  by $|W|$ we will denote the length of $W$. If $(W_i)_{i=1}^n$ is a sequence in $\mathrm{V}_{\mathrm{un}}^d(H;X)$   and $(x_i)_{i=1}^n\in X^n$   then for every $j\in [n]$, by $\prod_{i=j}^n W_i(x_i)$, we denote  the concatenation ${W_{j}(x_j)}^\smallfrown \dots ^\smallfrown W_n(x_n)$.

\begin{lem}\label{LM3.1} 
If $(H,X)$ has the $(E,d)$-UHJP then for every $n,r\in \mathbb{N}$ there exists a positive integer  $N=N(n,r)$ satisfying the following property. For any  $r$-coloring of $X^N$ there exists a sequence  $(W_i)_{i=1}^n$ in $\mathrm{V}_{\mathrm{un}}^d(H;X)$ with $\sum_{i=1}^n |W_i|=N$  such that for every $(x_i)_{i=1}^n \in X^n$, the set $\left\{\prod_{i=1}^n W_i(x_i'): \ \forall i\in [n] \, x_i' E x_i  \right\}$ is monochromatic.
\end{lem}

\begin{proof} 
For every $r\in\mathbb{N}$, let  $f(r)=N(H,X,E,d,r)$, where $N(H,X,E,d,r)$ is  as in  Definition \ref{D2}. It is clear that for $n=1$ we may set $N(1,r)=f(r)$ for all $r\in\mathbb{N}$. Assume that for some $n\in\mathbb{N}$ and all $r\in\mathbb{N}$, the numbers $N(n,r)$ have been defined.  Let $r\in\mathbb{N}$ be arbitrary and set
\begin{equation}\label{EQT2.2}
N(n+1,r)=N \left(n, r^{|X|}\right)+ f\left(r^{|X|^{\displaystyle N \big(n, r^{|X|}\big)}}\right).
\end{equation}

Now let $N=N(n+1,r)$ and let  $c:X^{N}\to [r]$ be an $r$-coloring of $X^N$. By \eqref{EQT2.2}, we have $N=N_1+N_2$ where  $N_1= N \left(n, r^{|X|}\right)$ and   $N_2=f(r^{|X|^{N_1}})$. Let  $\displaystyle c_2:X^{N_2}\to [r^{|X|^{N_1}}]$ defined by $\displaystyle c_2(\bm{y})=\left(c\left(\bm{x}^\smallfrown \bm{y}\right)\right)_{\bm{x}\in X^{N_1}}$, for every $\bm{y}\in X^{N_2}$. By the choice of $N_2$ there exists a variable word $W\in \mathrm{V}_{\mathrm{un}}^d(H;X)$ with $|W|=N_2$ such that
for every $\bm{x}\in X^{N_1}$,
\begin{equation}\label{EQT2.5}
c\left(\bm{x}^\smallfrown W(x')\right)=c\left(\bm{x}^\smallfrown W(x)\right) \text{ if } x' E x.
\end{equation}

Now define $\displaystyle c_1:X^{N_1}\to [r^{|X|}]$ by $\displaystyle c_1(\bm{x})=\left(c\left(\bm{x}^\smallfrown W(x)\right)\right)_{x\in X}$, for every $\bm{x}\in X^{N_1}$. By the choice of $N_1$ there exists a sequence
$(W_i)_{i=1}^n$ in $\mathrm{V}_{\mathrm{un}}^d(H;X)$  with $\sum_{i=1}^n |W_i|=N_1$ such  that   for every $x\in X$,
\begin{equation}\label{ht1}
c\left(\prod_{i=1}^n {W_i(x_i')}^\smallfrown W(x)\right)=c\left(\prod_{i=1}^n {W_i(x_i)}^\smallfrown W(x)\right) \text{ if }  x'_i E x_i  \text{ for all } i\in [n].
\end{equation}

We set $W_{n+1}=W$. Then $(W_i)_{i=1}^{n+1}$ is a sequence in  $\mathrm{V}_{\mathrm{un}}^d(H;X)$  with $\sum_{i=1}^{n+1} |W_i|=N_1+N_2=N$. Also, let   $(x_i)_{i=1}^{n+1}, (x'_i)_{i=1}^{n+1}\in X^{n+1}$ such that $x_i E x_i'$ for all $i\in [n+1]$. Then, $c\left(\prod_{i=1}^{n+1} {W_i(x_i)}\right) \stackrel{\eqref{ht1}}{=}c\left(\prod_{i=1}^n {W_i(x'_i)}^\smallfrown {W(x_{n+1})}\right) \stackrel{\eqref{EQT2.5}}{=} c\left(\prod_{i=1}^{n+1} W_i(x'_i)\right)$ and the proof is completed.
\end{proof}

\section{Some useful notation}\label{Sec4} 
We introduce here  some  notation which will facilitate our   proofs in the next sections. Fix for the following a finite group $G$ acting on a finite set $X$ and a nonempty subset $H$ of $G$.
\subsubsection{} 
Let $(W_i)_{i=1}^n$ be a sequence where each term  $W_i$ is either a constant or a uniform $H$-variable word  over $X$.  For a partition of $[n]=\bigcup_{j=1}^m F_j$, by the notation $ \prod_{j=1}^m \prod_{i\in F_j}W_i$ we  simply mean the concatenation ${W_1}^\smallfrown \dots ^\smallfrown W_n$.

\subsubsection{} 
Let $W=(w_i)_{i=1}^n\in \mathrm{V}_{\mathrm{un}}^d(H;X)$.
The product representation
\begin{equation}\label{EQT3.1}
W=\prod_{i\in F}x_i \times \prod_{h\in H}\prod_{i\in F_h} v_h.
\end{equation}

means that $F=\{i\in [n]: w_i=x_i\in X\}$ and  $F_h=\{i\in [n]: w_i=v_h\}$, for all $h\in H$. By \eqref{EQT1.1}, for any $x\in X$,  $W(x)$ will be represented as
\begin{equation}\label{EQT3.2}
W(x)=\prod_{i\in F}x_i \times \prod_{h\in H}\prod_{i\in F_h} hx.
\end{equation}

Also, for every  $\tau\in G$, by $W^\tau$ we will denote  the  $H\tau$-variable word over $G$  resulting from  $W$ by leaving its  letters unchanged and substituting each  variable $v_h$ by $v_{h\tau }$.  Notice that if $W$ is represented as in \eqref{EQT3.1} then  $W^\tau$ is represented as
\begin{equation}\label{EQT3.3}
  W^\tau=\prod_{i\in F}x_i \times \prod_{h\in H}\prod_{i\in F_h} v_{h \tau }.
\end{equation}

Using \eqref{EQT3.2} and \eqref{EQT3.3},  for every $\tau\in G$ and every  $x\in X$, we have $W^\tau(x)=\prod_{i\in F}x_i \times \prod_{h\in H}\prod_{i\in F_h} (h\tau) x=\prod_{i\in F}x_i \times \prod_{h\in H}\prod_{i\in F_h} h (\tau x)$, and hence,
\begin{equation}\label{EQT3.4}
W^\tau(x)=W(\tau x).
\end{equation}

Moreover notice that  if $H$ is a subgroup of $G$ and $\tau\in H$, then  $W^\tau\in \mathrm{V}_{\mathrm{un}}^d(H ;X)$.

\section{Proof of Proposition \ref{PRP2.2}}\label{Sec5} 
As we have already mentioned in Section \ref{Sec2}, for the proof of Proposition \ref{PRP2.2} we will use  a combinatorial argument  due to K\v{r}\'\i\v{z} \cite{Kr91}. This  argument  requires  a Ramsey-type result which, although is a direct  consequence of Ramsey's theorem \cite{Ra30},  we state it explicitly below  and  we give a self-contained proof.

Let $T:\mathbb{N}\times \mathbb{N}\to \mathbb{N}$ inductively defined by $T(1,r)=1$,  $T(2,r)=r+1$ and $T(n+1,r)=T(n,2^r)$ for every $n\geqslant 2$ and every $r\in\mathbb{N}$. Notice that  $T(n,r)$ is a  tower-type function, for example  $T(3,r)=2^r+1$, $T(4,r)=2^{2^r}+1$, $T(5,r)=2^{2^{2^r}}+1$ and so on. Given a set $X$ and  $m\in\mathbb{N}$, by
$\binom{X}{m}$  we denote the set of all $m$-element subsets  of $X$. Finally, for a nonempty  $X\subseteq \mathbb{N}$, we set 
\[
X_\ast=X\setminus\{\min X\}\ \ \  \text{ and  } \ \ \  X^\ast=X\setminus\{\max X\}.
\]

\begin{lem}\label{LM5.1}
Let $p,r\in\mathbb{N}$ with $p\geqslant 2$ and $n= T(p,r)$. Then for any $r$-coloring of $\binom{[n]}{p-1}$  there exists $P\subseteq [n]$ with $|P|=p$ such that $P_*$ and $P^*$ have the same color.
\end{lem}

\begin{proof}
If $p=2$ then $T(2,r)=r+1$ and the lemma follows easily  by the pigeonhole principle. We proceed by induction on $p$.  Assume  that the lemma is true for some  $p\geqslant 2$. Let  $q=p+1$, $n=T(q,r)$ and $c:\binom{[n]}{q-1}\to [r]$. We  look for a subset $Q\subseteq [n]$ with $|Q|=q$ and such that $c(Q_*)=c(Q^*)$. Let $\mathcal{P}([r])$ be the powerset of $[r]$ and let  $c':  \binom{[n]}{q-2}\to \mathcal{P}([r])$, defined by
\begin{equation}\label{EQT4.1}
c'(B)=\left\{c(B\cup\{x\}):\ x\in [n] \text{ and } \max B<x\right\},
\end{equation}

for every $B\in \binom{[n]}{q-2}$. We may view $c'$ as a $2^r$-coloring of $\binom{[n]}{q-2}=\binom{[n]}{p-1}$ and hence, since $n=T(q,r)=T(p+1, r)=T(p, 2^r)$, by our inductive assumption, we can find $P\in\binom{[n]}{p}$ such that $c'\left(P_{\ast}\right)=c'\left(P^{\ast}\right)$. Since $P=P^*\cup\max P$, we have $c(P)\in c'\left(P^{\ast}\right)$ and thus, $c(P)\in c'\left(P_{\ast}\right)$. Hence, by \eqref{EQT4.1} for $B=P_{\ast}$, there exists $x\in [n]$ such that
\begin{equation}\label{EQT4.2}
\max P_{\ast} < x \ \text{ and } \  c\left(P_{\ast}\cup\{x\}\right)=c(P).
\end{equation}

We set $Q=P\cup\{x\}$. Since $\max P=\max P_{\ast} < x$ we have that  $|Q|=p+1=q$. Moreover, notice that $c\left(Q_\ast\right)=c\left(P_{\ast}\cup\{x\}\right)\stackrel{\eqref{EQT4.2}}{=}c(P)=c\left(Q^\ast\right)$.
\end{proof}

Let $p\in\mathbb{N}$ with $p\geq 2$ and let $G=\{\tau^{j}:0\leqslant j\leqslant p-1\}$ be a cyclic group of order $p$. Proposition \ref{PRP2.2} is a consequence of the following lemma.

\begin{lem}\label{LM5.2} 
Let $k\in \{1,\dots,p-1\}$. Then for every  $r\in\mathbb{N}$ there exists a positive integer $N$ such that for any  $r$-coloring of $G^N$ there exists a variable word $W\in\mathrm{V}_{\mathrm{un}}^{p^k}(G;G)$ of length $N$ such that the set  $\{W(\tau^{j}):0\leqslant j\leqslant k\}$ is monochromatic.
\end{lem}

\begin{proof}
First let  $k=1$. Let $r\in\mathbb{N}$, $N=T(p,r)$ and $c:G^N\to [r]$. Then $c$ induces an $r$-coloring on $\binom{[N]}{p-1}$ as follows. For any $B=\{n_1<\cdots<n_{p-1}\}\in\binom{[N]}{p-1}$, we set $\bm{\tau}^{B}=(\tau_i^B)_{i=1}^N\in G^N$,  where
\begin{equation}\label{EQT4.3}
\tau^{B}_i=\begin{cases}	e &\text{if } i\notin B\\
						\tau^{q} 	&\text{if } i=n_q \text{ for some }  q\in \{1,\dots,p-1\}.
			\end{cases}
\end{equation}

Now let $\tilde{c}:\binom{[N]}{p-1}\to[r]$ defined by $\tilde{c}(B)=c\left(\bm{\tau}^{B}\right)$. Since $N=T(p,r)$,  by  Lemma \ref{LM5.1}, there exists $P\in\binom{[N]}{p}$ such that
\begin{equation}\label{EQT4.4}
\tilde{c}\left(P_\ast\right)=\tilde{c}\left(P^\ast\right).
\end{equation}

Without loss of generality, let  $P=[p]$ and let $W=\prod_{q=1}^pv_{\tau^{q-1}}\times\prod_{i=p+1}^N e$.
Clearly, $W\in \mathrm{V}_{\mathrm{un}}^{p}(G;G)$. Also, since $c\left(W(e)\right)	=c\left(\prod_{q=1}^p\tau^{q-1}\times\prod_{i=p+1}^N e\right)=c\left(\bm{\tau}^{P_{\ast}}\right)=\tilde{c}\left(P_{\ast}\right)$, and $c\left(W(\tau)\right)	=c\left(\prod_{q=1}^p\tau^{q}\times\prod_{i=p+1}^N e\right)=c\left(\bm{\tau}^{P^{\ast}}\right)=\tilde{c}\left(P^{\ast}\right)$, by \eqref{EQT4.4}, we get that $c\left(W(e)\right)=c\left(W(\tau)\right)$ and the proof for $k=1$ is completed.

Assume now that the lemma is true  for some $k\leqslant p-2$. By Definition \ref{D2}, our inductive assumption means that $G$ has the $(E_k,d_k)$-UHJP, where $d_k=p^k$ and $E_k$ is the equivalence relation on $G$ defined by
\[
g E_k g' \Leftrightarrow g,g'\in\{\tau^j: 0\leqslant j \leqslant  k\} \text{ or } g=g'.
\]

Let $r\in\mathbb{N}$ and  let  $N=N(n,r)$ be as in   Lemma \ref{LM3.1}, for $H=X=G$, $E=E_k$,   $d=d_k$ and $n=T(p,r^{k+1})$. Let $c: G^N\to [r]$ be an $r$-coloring of $G^N$.  By the choice of $N$, we can find  a sequence $(W_i)_{i=1}^n$ of variable words in  $\mathrm{V}_{\mathrm{un}}^{p^k}(G;G)$ with $\sum_{i=1}^n |W_i|=N$  such that for any $(g_i)_{i=1}^n \in G^n$ the set $\left\{\prod_{i=1}^n W_i(g_i'): \forall i\in [n] \ g_i' E_k g_i \right\}$ is monochromatic. For every $j\in\{0,\dots,k\}$  and every $B=\{n_1<\cdots<n_{p-1}\}\in\binom{[n]}{p-1}$, let  $\bm{\tau}^{B,j}=\left(\tau^{B,j}_i\right)_{i=1}^n\in G^n$, where
\begin{equation}\label{EQT4.5}
\tau^{B,j}_i=\begin{cases}	e 			&\text{if } i\notin B\\
								\tau^{q+j} 	&\text{if } i=n_q \text{ for some }q\in \{1,\dots,p-1\}.
			\end{cases}
\end{equation}

Now let  $\tilde{c}:\binom{[n]}{p-1}\to[r^{k+1}]$ defined  by $\tilde{c}(B)=\left(\tilde{c}_{j}(B)\right)_{j=0}^{k}$, where
\begin{equation}\label{EQT4.6}
\tilde{c}_{j}(B)=c\left(\prod_{i=1}^n W_i\left(\tau_i^{B,j}\right)\right).
\end{equation}

Since $n=T(p,r^{k+1})$, by Lemma \ref{LM5.1}, there exists $P\in\binom{[n]}{p}$ such that $\tilde{c}\left(P_\ast\right)=\tilde{c}\left(P^\ast\right)$, or equivalently,
\begin{equation}\label{EQT4.7}
\tilde{c}_{j}\left(P_\ast\right)=\tilde{c}_{j}\left(P^\ast\right), \text{ for all } j=0,\dots ,k.
\end{equation}

As in the case $k=1$, let us  assume that  $P=[p]$ and let
\begin{equation}\label{EQT4.8}
W=\prod_{q=1}^p W_q^{\tau^{q-1}}\times \prod_{i=p+1}^n W_i( e)
\end{equation}

where $W_q^{\tau^{q-1}}$ is as in \eqref{EQT3.3} for every  $q\in [p]$. Since $W_q\in \mathrm{V}_{\mathrm{un}}^{p^k}(G;G)$,  we have that $W_q^{\tau^{q-1}}\in \mathrm{V}_{\mathrm{un}}^{p^k}(G;G)$ for every $q\in [p]$ and consequently, $W\in \mathrm{V}_{\mathrm{un}}^{p^{k+1}}(G;G)$. It remains to show that  the set $\left\{W(\tau^j): 0\leqslant  j\leqslant k+1\right\}$ is monochromatic. Indeed, let $j\in\{0,\dots,k\}$. Then,
\[\begin{split}
c\left(W(\tau^{j+1})\right)	&\stackrel{\eqref{EQT3.4}, \eqref{EQT4.8}}{=}c\left(\prod_{q=1}^{p-1} W_q(\tau^{q+j})\times W_p(\tau^{j})\times \prod_{i=p+1}^n W_i( e)\right)\\
							&\stackrel{\tau^{j} E_k\, e}{=}c\left(\prod_{q=1}^{p-1} W_i(\tau^{q+j})\times W_p(e)\times  \prod_{i=p+1}^n W_i( e)\right)\\
							&\stackrel{\eqref{EQT4.5}}{=}c\left(\prod_{i=1}^n W_i\left(\tau^{P^{\ast},j}_i\right)\right)\stackrel{\eqref{EQT4.6}}{=}\tilde{c}_{j}(P^{\ast}).
\end{split}\]

Similarly,
\[\begin{split} c\left(W(\tau^{j})\right)&\stackrel{\eqref{EQT3.4}, \eqref{EQT4.8}}{=}c\left(W_1(\tau^{j})\times \prod_{q=2}^{p} W_q(\tau^{q-1+j})\times \prod_{i=p+1}^n W_i( e)\right)\\&\stackrel{\tau^{j}E_k\, e}{=}c\left(W_1(e)\times\prod_{q=2}^{p} W_q(\tau^{q-1+j})\times \prod_{i=p+1}^n W_i( e)\right)\\&\stackrel{\eqref{EQT4.5}}{=}c\left(\prod_{i=1}^n W_i(\tau^{P_{\ast},{j}}_i)\right)
\stackrel{\eqref{EQT4.6}}{=}\tilde{c}_{j}(P_{\ast}).
\end{split}
\]

By the above and \eqref{EQT4.7}, we get that $c\left(W(\tau^{j+1})\right)=c\left(W(\tau^{j})\right)$, for every $j\in \{0,\dots,k\}$, i.e. the set $\{W(\tau^{j}):0\leqslant j\leqslant k+1\}$ is monochromatic.
\end{proof}

\section{Proof of Proposition \ref{PRP2.4}}\label{Sec6} 
We fix for the following a finite group $G$ acting on a  finite set $X$ and  a subgroup $H$ of $G$ having  the $d$-UHJP for some $d\in\mathbb{N}$.

\begin{lem} \label{LM6.1} 
Let $y\in X$ and let $E_y$  be the  equivalence relation on $X$ defined by
\begin{equation}\label{EQT5.1}
x' E_y x\Leftrightarrow x',x\in H y \text{ or } x'=x.
\end{equation}

Then  $(H,X)$ has the $(E_y, d)$-UHJP.
\end{lem}

\begin{proof}
Let $r\in\mathbb{N}$  and since $H$ has the $d$-UHJP, let $N=N(H,d,r)$ be as in Definition \ref{D1} (for $H$ in place of $G$). Let  $c: X^N\to [r]$ and $c_H:H^N\to [r]$  defined by
\begin{equation}\label{EQT5.2}
c_H(h_1,\dots, h_N)=c\left(h_1 y,\dots, h_N y\right).
\end{equation}

By the choice of $N$ there exists $W_H\in \mathrm{V}_{\mathrm{un}}^d(H;H)$ of length $N$ such that
\begin{equation}\label{EQT5.3}
c_H\left(W_H(h)\right)=c_H\left(W_H(h')\right) \ \forall h,h'\in H.
\end{equation}

Writing $W_H$ as $W_H=\prod_{i\in F} h_i \times \prod_{h\in H}\prod_{i\in F_h}v_{h}$, we  define
\begin{equation}\label{EQT5.4}
 W=\prod_{i\in F} h_i y \times \prod_{h\in H}\prod_{i\in F_h}v_{h}.
\end{equation}

It is clear that $W\in\mathrm{V}_{\mathrm{un}}^d(H;X)$ and $|W|=N$. Hence, to complete the proof it remains to verify  that the set $\{W(x'): x' E_y x\}$ is monochromatic. Since $E_y$ is the identity relation on $X\setminus Hy$, it suffices to show that $c(W(x))=c(W(x'))$, for every $x,x'\in Hy$.   Indeed, let $x\in  Hy$ and
choose $h_x\in H$ such that $x=h_x y$. Then,
\[
\begin{split} c\left(W(x)\right)= c\left(W(h_xy)\right)&\stackrel{\eqref{EQT5.4}}{=}c\left(\prod_{i\in F} h_i y \times \prod_{h\in H}\prod_{i\in F_h}h(h_x y)\right)\\
&=c\left(\prod_{i\in F} h_i y \times \prod_{h\in H}\prod_{i\in F_h}(h h_x) y\right)\\
&\stackrel{\eqref{EQT5.2}}{=}c_H\left(\prod_{i\in F} h_i \times \prod_{h\in H}\prod_{i\in F_h}hh_x\right)
=c_H\left(W_H(h_x)\right),\end{split}
\]
and thus, by \eqref{EQT5.3}, the set $\{W(x) :x\in Hy\}$ is monochromatic.
\end{proof}

Let  $H y_1,\dots, H y_{p}$ be the orbits of $H$ in $X$ and for every $k\in [p]$, let $E_k$ be the  equivalence relation on $X$ defined by
\begin{equation}\label{EQT5.5}
x' E_k\, x\Leftrightarrow \exists j\in [k] \text{ such that } x',x\in H y_j \text{ or } x'=x.
\end{equation}

Notice that $E_{p}=E_{X|H}$ and hence, Proposition \ref{PRP2.4} follows by  the next lemma.

\begin{lem}\label{LM6.2}
For every $k\in [p]$, $(H,X)$ has the $(E_k, d^k)$-UHJP.
\end{lem}

\begin{proof}
The case $k=1$ has already been done in Lemma \ref{LM6.1}. We proceed by induction on $k\in [p]$. Let  $k\in [p-1]$  and assume  that  $(H,X)$ has the $(E_k, d^k)$-UHJP. Fix  $r\in\mathbb{N}$ and set  $n= N(H,X, E_k, d^k, r)$.    By Lemma \ref{LM6.1},   $(H,X)$  has the $(E_{y_{k+1}}, d)$-UHJP.  Let $N=N(n,r)$ be as in  Lemma  \ref{LM3.1} (for $E=E_{y_{k+1}}$).

Let $c:X^N\to [r]$. By the choice of $N$,  there exists a  sequence
$(W_i)_{i=1}^{n}$ in $\mathrm{V}_{\mathrm{un}}^d(H;X)$ with $\sum_{i=1}^{n}|W_i|=N$ and such that
\begin{equation}\label{EQT5.6}
c\left(\prod_{i=1}^{n} W_i(x_i')\right)=c\left(\prod_{i=1}^{n} W_i(x_i)\right)  \text{ whenever  } x'_i E_{y_{k+1}}  x_i \   \forall i\in [n].
\end{equation}

Let  $c': X^n\to [r]$ defined  by
\begin{equation}\label{EQT5.7}
c'(x_1,\dots,x_{n}) =c\left(\prod_{i=1}^{n} W_i(x_i)\right).
\end{equation}

By the choice of $n$ and \eqref{EQT5.5},  there exists   a variable word $W'\in\mathrm{V}_{\mathrm{un}}^{d^{k}}(H;X)$ with $|W'|=n$  such that
\begin{equation}\label{EQT5.8}
c'\left(W'(x)\right)=c'\left(W'(y_j)\right), \text{ for  every } j\in [k] \text{  and every } x\in Hy_j.
\end{equation}

Let  $W'=\prod_{i\in F}x_i \times \prod_{h\in H}\prod_{i\in F_h} v_h$, and let
\begin{equation}\label{EQT5.9}
W=\prod_{i\in F}W_i(x_i)\times \prod_{h\in H}\prod_{i\in F_h}W_i^h.
\end{equation}

Since $W_i^h\in \mathrm{V}_{\mathrm{un}}^{d}(H;X)$,  for every  $i\in [n]$  and $|H|\cdot |F_h|=d^k$, for every $h\in H$, it is easy to check that $W$ is a uniform  $H$-variable word over $X$ of  degree $d^{k+1}$. Also, it is clear that  $|W|=\sum_{i=1}^{n}|W_i|=N$. By \eqref{EQT5.5}, to complete the proof it remains to show that for every $j\in [k+1]$, the set $\{W(x): x\in Hy_j\}$ is  monochromatic. First, let  $j=k+1$ and let $x\in Hy_{k+1}$. Then $x=h_x y_{k+1}$, for some $h_x\in H$ and 
\[
\begin{split}
c(W(x))=c\left(W(h_x y_{k+1})\right)& \stackrel{\eqref{EQT5.9}}{=}c\left(\prod_{i\in F}W_i(x_i)\times \prod_{h\in H}\prod_{i\in F_h}W_i\left(h(h_x y_{k+1})\right)\right)\\
&\stackrel{\eqref{EQT5.6}}{=}c\left(\prod_{i\in F}W_i(x_i)\times \prod_{h\in H}\prod_{i\in F_h}W_i\left(hy_{k+1}\right)\right)\\
&\stackrel{\eqref{EQT5.9}}{=}c\left(W(y_{k+1})\right).\end{split}
\]
Also, if $j\in [k]$ and  $x=h_x y_{j}\in Hy_{j}$, then,
\[
\begin{split}
c(W(x))=c\left(W(h_xy_j)\right)&\stackrel{\eqref{EQT5.9}}{=}c\left(\prod_{i\in F}W_i(x_i)\times \prod_{h\in H}\prod_{i\in F_h}W_i\left(h (h_x y_j)\right)\right)\\&\stackrel{\eqref{EQT5.7}}{=}
c'\left(\prod_{i\in F}x_i\times \prod_{h\in H}\prod_{i\in F_h}h (h_x y_j)\right)\\&
=c'\left(W'(h_x y_j)\right)\stackrel{\eqref{EQT5.8}}{=}c'\left(W'(y_j)\right).\end{split}
\] 
By the above the proof is completed.
\end{proof}

\section{Proof of Proposition \ref{PRP2.6}}\label{Sec7}
Let $H$, $G$, $K$ be finite groups such that $G$ is an extension of $K$ by $H$. Let    $\pi: G\to K$ be the a surjective homomorphism and $\iota: H\to G$ be the  injective homomorphism such that the image of $\iota$ equals the kernel of $\pi$. Assume that   $H$  has the $d_H$-UHJP. Identifying $H$ with  $\iota (H)$, by Corollary \ref{CRL2.5}, we have that  $(H,G)$ has the  $(E_{G|H}, d_H^p)$-UHJP, where $p=|G/H|$. Hence, Proposition \ref{PRP2.6} is a direct consequence of the above and the following lemma.

\begin{lem}\label{LM7.1}  
Let $H, G, K$ be finite groups such that there exists a surjective homomorphism  $\pi:G\to K$  with kernel $H$. If   $(H,G)$ has the  $(E_{G|H}, d_1)$-UHJP  for some $d_1\in \mathbb{N}$ and $K$ has the $d_K$-UHJP then $G$ has the  $d$-UHJP, where $d=d_1 d_K$.
\end{lem}

\begin{proof} 
Notice that  $H$ is a normal subgroup of $G$ and
\begin{equation}\label{EQT6.1}
g  E_{G|H} g' \Leftrightarrow \pi(g)=\pi(g').
\end{equation}

Let $r\in\mathbb{N}$ and let $N=N(n,r)$ be as in Lemma \ref{LM3.1}, for $X=G$, $E=E_{G|H}$,  $d=d_1$  and  $n=N(K,d_K,r)$. Let $c:G^N\to [r]$. By the choice of $N$ and \eqref{EQT6.1},  there exists a sequence $(W_i)_{i=1}^{n}$ of variable words in $\mathrm{V}_{\mathrm{un}}^{d_1}(H;G)$, with $\sum_{i=1}^{n} |W_i|=N$, and such that
\begin{equation}\label{EQT6.2}
c\left(\prod_{i=1}^{n} W_i(g_i)\right)=c\left(\prod_{i=1}^{n} W_i(g'_i)\right) \text{ whenever } \pi(g_i)=\pi(g_i') \, \forall \, i\in[n].
\end{equation}

Let $c_K:K^{n}\to [r]$ defined by
\begin{equation}\label{EQT6.3}
c_K(\kappa_1,\dots,\kappa_{n})=c\left(\prod_{i=1}^{n} W_i(g_i)\right) \text{ if }  \kappa_i=\pi(g_i) \, \forall \, i\in[n].
\end{equation}

By \eqref{EQT6.2} and since $\pi:G\to K$ is onto,  the coloring $c_K$ is well-defined. By the choice of  $n$,  there exists a  variable word $W_K\in\mathrm{V}_{\mathrm{un}}^{d_K}(K;K)$ of length $n$ and   such that
\begin{equation}\label{EQT6.4}
c_K\left(W_K(\kappa)\right)=c_K\left(W_K(\kappa')\right)\, \forall \kappa, \kappa'\in K.
\end{equation}

Let $W_K=\prod_{i\in F}\kappa_i \times \prod_{\kappa\in K}\prod_{i\in F_\kappa} v_\kappa$ and set
\begin{equation}\label{EQT6.5}
W=\prod_{i\in F}W_i(g_{\kappa_i}) \times \prod_{\kappa\in K}\prod_{i\in F_\kappa}W_i^{g_\kappa},
\end{equation}

where for every $\kappa\in K$, $g_\kappa\in G$ is such that $\pi(g_\kappa)=\kappa$. It is easy to see  that $|W|=\sum_{i=1}^{n} |W_i|=N$. Moreover,  $\bigcup_{\kappa\in K} Hg_\kappa=G$ and since for every $\kappa\in K$,    $|K|\cdot |F_\kappa|=d_K$  and $W_i^{g_\kappa}\in \mathrm{V}_{\mathrm{un}}^{d_1}(H g_\kappa;G)$, we conclude that $W\in \mathrm{V}_{\mathrm{un}}^{d}(G;G)$, for $d=d_1 d_K$. Finally, let  $g\in G$ and let $\kappa_g=\pi(g)$. Then,
\[
\begin{split}c\left(W(g)\right) & \stackrel{\eqref{EQT6.5}}{=}c\left(\prod_{i\in F}W_i(g_{\kappa_i}) \times \prod_{\kappa\in K}\prod_{i\in F_\kappa}W_i\left(g_\kappa g\right)\right)
\\&\stackrel{ \eqref{EQT6.3}}{=}c_K\left(\prod_{i\in F}\kappa_i\times \prod_{\kappa\in K}\prod_{i\in F_\kappa}\pi\left(g_\kappa g\right) \right)
\\&=c_K\left(\prod_{i\in F}\kappa_i\times \prod_{\kappa\in K}\prod_{i\in F_\kappa}\pi(g_\kappa)\pi(g)\right)
\\&=c_K\left(\prod_{i\in F}\kappa_i\times \prod_{\kappa\in K}\prod_{i\in F_\kappa}\kappa \kappa_g\right)=c_K\left(W_K(\kappa_g)\right),
\end{split}
\]
and hence, by \eqref{EQT6.4}, the set  $\{W(g): g\in G\}$ is monochromatic.
\end{proof}

\section{Notes and remarks}\label{Sec8}
The class of finite groups having the $d$-UHJP for some $d\in\mathbb{N}$ shares similar properties with the class of solvable groups, namely, by propositions \ref{PRP2.2} and  \ref{PRP2.6}, it contains all finite  cyclic groups and it  is closed under  extensions. It is well known that the class of solvable groups is also closed under subgroups and quotients. The next proposition says that the same holds for the finite groups with the $d$-UHJP.

\begin{prop}\label{PRP8.1}
Let  $G$ be a finite group having   the $d$-UHJP for some $d\in\mathbb{N}$. Then  the following are satisfied. (a) If $H$ is a subgroup of $G$  then $H$ has the $d$-UHJP.  (b) If $H$ is a normal subgroup of $G$ then $G/H$ has the $d$-UHJP.
\end{prop}

\begin{proof} (a) Let  $H$  be a subgroup of $G$. Let $p$ be the  index of $H$ in $G$ and  choose  $\tau_1,\dots,\tau_{p-1}\in G$ such that $G/H=\{ H, \tau_1 H, \dots, \tau_{p-1}H\}$. Setting $\tau_0=e$,  notice that  for every $g\in G$ there exists a unique pair $(i_g, h_g)\in \{0,\dots,p-1\}\times H$ such that $g=\tau_{i_g} h_g$. Let $\varphi:G\to H$  defined by  $\varphi(g)=h_g$ for all $g\in G$. It is easy to see that   $\varphi$ satisfies the following properties. (i) It  is surjective  and $\varphi(h)=h$ for all $h\in H$, (ii) for every  $g\in G$ and every  $h\in H$, $\varphi(gh)=\varphi(g) h$, and  (iii) for every  $h\in H$, $|\{g\in G: \varphi(g)=h\}|=p$.

Now fix $r\in\mathbb{N}$ and let $c: H^N\to [r]$, where $N=N(G,d,r)$ is as in Definition \ref{D1}. Let $\tilde{c}:G^N\to [r]$ defined by  $\tilde{c}(g_1,\dots,g_N)=c\left(\varphi(g_1),\ldots,\varphi(g_N)\right)$. By the choice of $N$, we  can find  $\tilde{W}\in \mathrm{V}_{\mathrm{un}}^d(G;G)$ of length $N$ and $k_0\in [r]$  such that $\tilde{c}\left(\tilde{W}(g)\right)=k_0$ for all $g\in G$. Let $\tilde{W}=\prod_{i\in F} g_i \times \prod_{g\in G}\prod_{i\in F_g} v_g$ and set $W=\prod_{i\in F} h_i \times \prod_{h\in H}\prod_{i\in F'_h} v_h$, where $h_i=\varphi(g_i)$, for all $i\in F$ and  $F'_h=\cup\{F_g: \varphi(g)=h\}$ for all $h\in H$.  Since $\tilde{W}\in\mathrm{V}_{\mathrm{un}}^d(G;G)$, by (iii)   we conclude  that $W\in \mathrm{V}_{\mathrm{un}}^d(H;H)$.
It remains to show that  $\{W(h):h\in H\}$ is $c$-monochromatic. Indeed, let  $h\in H$. Then,
$k_0=\tilde{c}\left(\tilde{W}(h)\right)=
c\left(\prod_{i\in F} \varphi(g_i) \times \prod_{g\in G}\prod_{i\in F_g} \varphi(g h)\right)\stackrel{(ii)}{=}
c\left(\prod_{i\in F} h_i \times \prod_{g\in G}\prod_{i\in F_g} \varphi(g) h\right)=c\left(W(h)\right)$,
and the proof  is  completed. (b) The proof is similar to the above by  using the surjective homomorphism  $g\to gH$ from $G$ to $G/H$ in place of $\varphi$.
\end{proof}

A natural question is  whether there exists a non-solvable group having the $d$-UHJP for some $d\in\mathbb{N}$. The  first candidate groups here are  the alternating  group $A_5$ or  the symmetric group $S_5$. A more general question is the following.

\begin{question}
\emph{Let $G$ be a finite group. Suppose that $G$ contains two subgroups $H$ and $K$ such that $H\cap K=\{e\}$ and $G=HK=\{hk:h\in H \text{ and } k\in K\}$. If $H$ has the $d_H$-UHJP and $K$ has the $d_K$-UHJP, is it true that $G$ must have the $d$-UHJP for some $d\in\mathbb{N}$}?
\end{question}

Notice that, if we additionally assume that $H$ is normal in $G$, then $G$ is an extension of $K$ by $H$ and the conclusion holds by Proposition \ref{PRP2.6}. However, if we drop the assumption of normality, then  the proof of Proposition \ref{PRP2.6} (see  Lemma \ref{LM7.1}) cannot be carried out.

An affirmative answer to the above question would have as a consequence that every finite group has the $d$-UHJP for some $d\in\mathbb{N}$. To see this, notice that by Proposition \ref{PRP8.1} and Cayley's theorem, it is enough to show that every $S_n$ has the $d_n$-UHJP for some $d_n\in\mathbb{N}$. Indeed,  $S_{n}=H_n C_n$, where $H_n$ is the set of all permutations on  $[n]$  which stabilize a fixed $i\in [n]$ and $C_n$ is the cyclic group  generated by the cycle $(1\, 2\,\dots n)$. By Proposition \ref{PRP2.2}, $C_n$ has the $n^{n-1}$-UHJP and clearly $H_n$ is isomorphic to $S_{n-1}$. By induction the conclusion follows.

The second author would like to thank the organization of ``Ramsey DocCourse Prague 2016'', which took place at Charles University in Prague, and during which, among others, he made his first contact with the area of Euclidean Ramsey theory and motivated him to investigate the problems presented in this work.

\end{document}